\documentclass{article}
\usepackage{amssymb,amsthm,amsmath,latexsym}
\usepackage{authblk}
\usepackage[symbol]{footmisc}
\def\correspondingauthor{\footnote{Corresponding author.}}
\newtheorem {theorem}{Theorem}
\newtheorem {corollary}{Corollary}

\newtheorem {definition}{Definition}

\newtheorem {remark}{Remark}
\numberwithin{equation}{section}

\begin{document}
\title{A $q$-analogue of $ \bar{\boldsymbol{\alpha}}$-Whitney Numbers}
\author[1]{B. S. El-Desouky}
\author[1]{F. A. Shiha\correspondingauthor{}}
\affil[1]{ Department of Mathematics, Faculty of Science, Mansoura University, 35516 Mansoura, Egypt}

\maketitle
\begin{abstract}
We define the $(q,\bar{\boldsymbol{\alpha}})$-Whitney numbers which are reduced to the $\bar{\boldsymbol{\alpha}}$-Whitney
numbers when $q\rightarrow1$. Moreover, we obtain several  properties of these numbers such as explicit formulas, recurrence relations, generating functions, orthogonality and inverse relations. Finally, we define the $\bar{\boldsymbol{\alpha}}$-Whitney-Lah
numbers as a generalization of the $r$-Whitney-Lah numbers and we introduce their important basic properties. \\


\end{abstract}
\section {Introduction}
El-Desouky et al. \cite{beih16} introduced the $\bar{\boldsymbol{\alpha}}$-Whitney numbers of both kinds  as a new family of numbers generalizing many types of numbers such as $r$-Whitney numbers, Whitney numbers, $r$-Stirling numbers, Jacobi-Stirling numbers and Legendre-Stirling numbers.

The $\bar{\boldsymbol{\alpha}}$-Whitney numbers of the first kind $ w_{m,\bar{\boldsymbol{\alpha}}}(n,k)$ and  second kind $W_{m,\bar{\boldsymbol{\alpha}}}(n,k)$ are defined by
\begin{equation}\label{E:wit1}
(x;\bar{\boldsymbol{\alpha}}|m)_n=\sum_{k=0}^nw_{m,\bar{\boldsymbol{\alpha}}}(n,k) x^k,
\end{equation}
and
\begin{equation}
   x^n= \sum_{k=0}^{n} W_{m,\bar{\boldsymbol{\alpha}}}(n,k)(x;\bar{\boldsymbol{\alpha}}|m)_k,
\end{equation}
where $\bar{\boldsymbol{\alpha}}=(\alpha_0 ,\alpha_1 , \ldots , \alpha_{n-1} ) $, and
\[
(x;\bar{\boldsymbol{\alpha}}|m)_n=\prod_{j=0}^{n-1}(x-\alpha_j-jm)\: \: \text{with}\: (x;\bar{\boldsymbol{\alpha}}|m)_0=1.
\]
The $\bar{\boldsymbol{\alpha}}$-Whitney numbers of the first and second kind satisfying recurrence relations of the form:
\begin{equation*}
 w_{m,\bar{\boldsymbol{\alpha}}}(n+1,k)=w_{m,\bar{\boldsymbol{\alpha}}}(n,k-1)-(\alpha_n+nm)w_{m,\bar{\boldsymbol{\alpha}}}(n,k),
 \end{equation*}
  \begin{equation*}
 W_{m,\bar{\boldsymbol{\alpha}}}(n+1,k)=W_{m,\bar{\boldsymbol{\alpha}}}(n,k-1)+(\alpha_k+km)W_{m,\bar{\boldsymbol{\alpha}}}(n,k).
 \end{equation*}
 Note that the $\bar{\boldsymbol{\alpha}}$-Whitney numbers coincide with the $r$-Whitney numbers and Whitney numbers by setting $\bar{\boldsymbol{\alpha}}=(r, r, \ldots, r)$ and $\bar{\boldsymbol{\alpha}}=(1,1,\ldots,1)$, respectively.
 Many properties of the $\bar{\boldsymbol{\alpha}}$-Whitney numbers, $r$-Whitney numbers and Whitney numbers can be found in  \cite{beih16, cheon12, ben96, mango15, merca14, mezo2010, mezo15}.

The organization of this article is as follows. In the next two sections, we define the $q$-analogue of the $\bar{\boldsymbol{\alpha}}$-Whitney
numbers of the first and second kind denoted by $w_{q, m,\bar{\boldsymbol{\alpha}}}(n,k)$ and $W_{q, m,\bar{\boldsymbol{\alpha}}}(n,k)$, respectively, and obtain their recurrence relations, explicit formulas and generating functions. In the third section, we obtain the orthogonality property of the both kinds of the $(q,\bar{\boldsymbol{\alpha}})$-Whitney numbers which yields to the inverse relations. Moreover we give some important special cases. In the fourth section, we define the $\bar{\boldsymbol{\alpha}}$-Whitney-Lah numbers and deduce its recurrence relation, explicit formula and matrix representation.

Let $ 0<q<1$, $x$ a real number, $ [x]_q=1+q+\cdots+q^{x-1}$ and $[0]_q=0$. The $q$-factorial of $x$ is defined by $[x]_q!=[x]_q[x-1]_q\cdots [2]_q[1]_q$, and the $q$-falling factorial of order $n$ is defined by
\[
 ([x]_q)_n=\prod_{j=0}^{n-1}[x-j]_q \:\text{and}\: ([x]_q)_o=1.
\]

Moreover, the following definitions and notation are introduced.
\[
([x;\bar{\boldsymbol{\alpha}}]_q)_n=\prod_{j=0}^{n-1}[x-\alpha_j]_q=[x-\alpha_0]_q[x-\alpha_1]_q\cdots[x-\alpha_{n-1}]_q \: \text{with}\: ([x;\bar{\boldsymbol{\alpha}}]_q)_0=1,
\]
\[
([x;\bar{\boldsymbol{\alpha}}|m]_q)_n=\prod_{j=0}^{n-1}[x-\alpha_j-jm]_q, \: \text{with} \: ([x;\bar{\boldsymbol{\alpha}}|m]_q)_0=1,
\]
and
\[
\big\langle[x;\bar{\boldsymbol{\alpha}}|m]_q \big \rangle_n=\prod_{j=0}^{n-1}([x]_q-[\alpha_j+jm]_q), \: \text{with}\: \big \langle[x;\bar{\boldsymbol{\alpha}}|m]_q \big \rangle_0=1.
\]
\section {The $(q,\bar{\boldsymbol{\alpha}})$-Whitney numbers of the first kind }
\begin{definition}\label{D:wi1}
The $(q,\bar{\boldsymbol{\alpha}})$-Whitney numbers of the first kind $ w_{q,m,\bar{\boldsymbol{\alpha}}}(n,k)$ are defined by
\begin{equation}\label{E:wi1}
\big\langle[x;\bar{\boldsymbol{\alpha}}|m]_q \big \rangle_n=\sum_{k=0}^nw_{q,m,\bar{\boldsymbol{\alpha}}}(n,k) [x]^k_q,
\end{equation}
where $ w_{q,m,\bar{\boldsymbol{\alpha}}}(0,0)=1$ and $ w_{q,m,\bar{\boldsymbol{\alpha}}}(n,k)=0$ for $ k>n$ or $k<0$.
\end{definition}
Since for the q-numbers we have $[x-y]_q=q^{-y}([x]_q-[y]_q)$. Then
\[
([x;\bar{\boldsymbol{\alpha}}|m]_q)_n=q^{-\sum_{i=0}^{n-1}\alpha_i+im}\prod_{j=0}^{n-1}([x]_q-[\alpha_j+jm]_q).
\]
Thus Eq.~(\ref{E:wi1}) in Definition \ref{D:wi1} can be written in the equivalent form
\[
([x;\bar{\boldsymbol{\alpha}}|m]_q)_n= q^{-\sum_{i=0}^{n-1}\alpha_i+im}\sum_{k=0}^nw_{q,m,\bar{\boldsymbol{\alpha}}}(n,k) [x]^k_q.
\]
In particular, note that $w_{q,m,\bar{\boldsymbol{\alpha}}}(n,k)$ is reduced to the  $w_{m,\bar{\boldsymbol{\alpha}}}(n,k)$ when $q\rightarrow1$.
\begin{theorem}
The $(q,\bar{\boldsymbol{\alpha}})$-Whitney numbers of the first kind satisfy the recurrence relation
 \begin{equation} \label{E:rec1}
 w_{q,m,\bar{\boldsymbol{\alpha}}}(n+1,k)=w_{q,m,\bar{\boldsymbol{\alpha}}}(n,k-1)-[\alpha_n+nm]_qw_{q,m,\bar{\boldsymbol{\alpha}}}(n,k),
 \end{equation}
 where $n\geq k\geq 1$, and
 \begin{equation} \label{E:rec12}
 w_{q,m,\bar{\boldsymbol{\alpha}}}(n,0)=(-1)^{n}\prod_{i=0}^{n-1}[\alpha_i+im]_q.
 \end{equation}
\end{theorem}
\begin{proof}Since $\big\langle[x;\bar{\boldsymbol{\alpha}}|m]_q \big\rangle_{n+1}=\big\langle[x;\bar{\boldsymbol{\alpha}}|m]_q\big\rangle_{n}\:([x]_q-[\alpha_n+nm]_q)$.

Using Eq.~(\ref{E:wi1}), we get
\begin{equation*}
\begin{split}
&\sum_{k=0}^{n+1} w_{q,m,\bar{\boldsymbol{\alpha}}}(n+1,k)[x]^k_q = \big\langle[x;\bar{\boldsymbol{\alpha}}|m]_q \big\rangle_n\:([x]_q-[\alpha_n+nm]_q)\\
&=\sum_{k=0}^n w_{q,m,\bar{\boldsymbol{\alpha}}}(n,k)[x]_q^k\:([x]_q-[\alpha_n+nm]_q)\\
&=\sum_{k=0}^n w_{q,m,\bar{\boldsymbol{\alpha}}}(n,k)[x]_q^{k+1}-[\alpha_n+nm]_q \sum_{k=0}^n w_{q,m,\bar{\boldsymbol{\alpha}}}(n,k)[x]_q^k\\
&=\sum_{k=1}^{n+1} w_{q,m,\bar{\boldsymbol{\alpha}}}(n,k-1)[x]_q^k-[\alpha_n+nm]_q\sum_{k=0}^n w_{q,m,\bar{\boldsymbol{\alpha}}}(n,k) [x]_q^k.
\end{split}
\end{equation*}
Equating the coefficients of $[x]_q^k$ on both sides yields \eqref{E:rec1}.

For $k=0$, we find
 \[
  w_{q,m,\bar{\boldsymbol{\alpha}}}(n+1,0)= - [\alpha_n+nm]_q\,w_{q,m,\bar{\boldsymbol{\alpha}}}(n,0) \: ,\: n=0 , 1, 2, \ldots,
 \]
 successive application gives \eqref{E:rec12}.
 \end{proof}
  \begin {definition} The $(q,\bar{\boldsymbol{\alpha}})$-Whitney matrix of the first kind is the $n\times n$ lower triangular matrix defined by
   \[
 \mathbf{W}_{1}:= \mathbf{w}_{q,m,\bar{\boldsymbol{\alpha}}}(n):=\big(w_{q,m,\bar{\boldsymbol{\alpha}}}(i,j)\big)_{0 \leq i,j\leq n-1}.
   \]
   \end{definition}
  For example when $n=4$ the matrix $ \mathbf{w}_{q,m,\bar{\boldsymbol{\alpha}}}(n)$ is given by
 \[
 \left(
 \begin{smallmatrix}
 1 &0 &0 &0\\ \\
- [\alpha_0]_q & 1 &0 &0 \\ \\
 [\alpha_0]_q[\alpha_1+m]_q &-[\alpha_0]_q-[\alpha_1+m]_q & 1&0  \\ \\
-[\alpha_0]_q[\alpha_1+m]_q[\alpha_2+2m]_q& [\alpha_0]_q [\alpha_1+m]_q +[ \alpha_0+\alpha_1+ m]_q[\alpha_2+2m]_q &-[\alpha_0]_q-[\alpha_1+m]_q-[\alpha_2+2m]_q&1\\
\end{smallmatrix}
\right)
\]
In particular, we note that $\mathbf{w}_{q,m,\bar{\boldsymbol{\alpha}}}(n)$ is reduced to the $\bar{\boldsymbol{\alpha}}$-Whiyney matrix of the first kind
 \cite{beih16} when $q\rightarrow 1$. In addition at $q\rightarrow 1$ and $\bar{\boldsymbol{\alpha}}=(r, r, \ldots, r)$ the $\mathbf{w}_{q,m,\bar{\boldsymbol{\alpha}}}(n)$ is reduced to the $r$-Whitney matrix of the first kind \cite{mezo15}.

Mansour et al. \cite{mansour12} derived a closed formula for all sequences satisfying
 a certain recurrence relation as follows:
\begin{theorem} \label{th:mans}\textnormal{ \cite[Theorem \ 1.1]{mansour12}}.  Suppose $(a_i)_{i\geq 0}$ and $(b_i)_{i\geq 0}$ are sequences of numbers with  $b_i\neq b_j$ when $i \neq j$ and
\begin{equation}
u(n,k)=u(n-1,k-1)+(a_{n-1}+b_k)u(n-1,k),
\end{equation}
with boundary conditions $u(n,0)=\prod_{i=0}^{n-1}(a_i+b_0)$ and $u(0,k)=\delta_{0k}$, where $\delta_{jk}$ is the Kronecker delta function, then
\begin{equation} \label{E:mex}
u(n,k)=\sum_{j=0}^k \left( \frac{\prod_{i=0}^{n-1}(b_j+a_i)}{\prod_{i=0 ,i\ne j}^k(b_j-b_i)}\right),\: \:\:\forall\: n,k \in \mathbb{N}.
\end{equation}
\end{theorem}
\begin{remark} \label{rem:mans} \textnormal{\cite[p.\ 25]{mansour12}}
\begin{enumerate}
\item The recurrence for $u(n,k)$ is given by
\begin{equation}\label{E:mrec}
u(n,k)=\sum_{j=k}^n u(j-1,k-1)\prod_{i=j}^{n-1}(a_i+b_k).
\end{equation}
\item In the case when $b_i=0$ for all $i$, then $u(n,k)$ is the $(n-k)$th elementary symmetric function of $a_0,\;a_1,\;\ldots,\;a_{n-1}$.
The elementary symmetric function $\sigma_k$ is defined by
 \[
 \sigma_k(z_1 , z_2 , \ldots , z_n)=\sum_{1 \leq j_1 < j_2 < \cdots < j_k \leq n}\: \prod_{i=1}^k z_{j_i},
 \]
 where $\sigma_0 =1$ and $\sigma_k=0$ when $n<k$ or $k<0$.
 \end{enumerate}
\end{remark}
\begin{theorem}
The $(q,\bar{\boldsymbol{\alpha}})$-Whitney numbers of the first kind are given by
\begin{equation}\label{E:ew1}
\begin{split}
w_{q,m,\bar{\boldsymbol{\alpha}}}(n,k) &=(-1)^{n-k}\: \sigma_{n-k}([\alpha_0]_q,\:[\alpha_1+m]_q,\:\ldots,\:[\alpha_{n-1}+(n-1)m]_q)\\
 &=(-1)^{n-k}\: \sum_{ 0 \leq j_1 < \cdots <j_{n-k} \leq n-1}\: \prod_{i=1}^{n-k}[\alpha_{j_i}+j_i\:m]_q.\\
\end{split}
 \end{equation}
and the following recurrence relation holds:
\begin{equation}\label{E:rrec1}
w_{q,m,\bar{\boldsymbol{\alpha}}}(n,k)= (-1)^{n-k}\sum_{j=k}^{n} w_{q,m,\bar{\boldsymbol{\alpha}}}(j-1,k-1) \prod_{i=j}^{n-1}[\alpha_i+im]_q.
\end{equation}
\end{theorem}
\begin{proof}
Taking $a_i=-[\alpha_i+im]_q$ and $b_i=0$, one can use Remark \ref{rem:mans} to obtain Eq.~(\ref{E:ew1})
 and Eq.~(\ref{E:rrec1}).
\end{proof}
From Eq.~(\ref{E:ew1}), we can obtain the generating function
\begin{equation}
\sum_{k=0}^n w_{q,m,\bar{\boldsymbol{\alpha}}}(n,k)x^k=\prod_{j=0}^{n-1}(x-[\alpha_j+jm ]_q).
\end{equation}
As $q \rightarrow1$, Eq.~(\ref{E:rrec1}) reduces to a new recurrence relation for the $\bar{\boldsymbol{\alpha}}$-Whitney numbers of the first kind given by
\begin{equation*}
w_{m,\bar{\boldsymbol{\alpha}}}(n,k)= (-1)^{n-k}\sum_{j=k}^{n} w_{m,\bar{\boldsymbol{\alpha}}}(j-1,k-1) \prod_{i=j}^{n-1}(\alpha_i+im).
\end{equation*}
\section {The $(q,\bar{\boldsymbol{\alpha}})$-Whitney numbers of the second kind }
\begin{definition}\label{D:wi2}
The $(q,\bar{\boldsymbol{\alpha}})$-Whitney numbers of the second kind $ W_{m,\bar{\boldsymbol{\alpha}}}(n,k)$ are defined by
\begin{equation}\label{E:wi2}
 [x]^n_q=\sum_{k=0}^n  W_{q,m,\bar{\boldsymbol{\alpha}}}(n,k) \big\langle[x;\bar{\boldsymbol{\alpha}}|m]_q \big\rangle_k ,
\end{equation}
where $ W_{q,m,\bar{\boldsymbol{\alpha}}}(0,0)=1$ and $ W_{q,m,\bar{\boldsymbol{\alpha}}}(n,k)=0$ for $ k>n$ or $k<0$.
\end{definition}
We notice that Eq.~(\ref{E:wi2}) in Definition \ref{D:wi2} can be written in the equivalent form
\[
[x]^n_q=\sum_{k=0}^n q^{\sum_{i=0}^{k-1}\alpha_i+im} W_{q,m,\bar{\boldsymbol{\alpha}}}(n,k)([x;\bar{\boldsymbol{\alpha}}|m]_q)_k.
\]
\begin{theorem}
The $(q,\bar{\boldsymbol{\alpha}})$-Whitney numbers of the second kind satisfy the recurrence relation
 \begin{equation} \label{E:rec2}
 W_{q,m,\bar{\boldsymbol{\alpha}}}(n+1,k)=W_{q,m,\bar{\boldsymbol{\alpha}}}(n,k-1)+[\alpha_k+km]_qW_{q,m,\bar{\boldsymbol{\alpha}}}(n,k),
 \end{equation}
 where $n\geq k\geq 1$, for $k=0$ we have
 \begin{equation} \label{E:rec21}
 W_{q,m,\bar{\boldsymbol{\alpha}}}(n,0)=[\alpha_0]_q^n.
 \end{equation}
\end{theorem}
\begin{proof} Since we have $[x]_q^{n+1}=[x]_q^n\:([x]_q-[\alpha_k+km]_q+[\alpha_k+km]_q).$
Using \eqref{E:wi2}, we get
\begin{equation*}
\begin{split}
&\sum_{k=0}^{n+1} W_{q,m,\bar{\boldsymbol{\alpha}}}(n+1,k)\big\langle[x;\bar{\boldsymbol{\alpha}}|m]_q \big\rangle_k \\
&=\sum_{k=0}^n  W_{q,m,\bar{\boldsymbol{\alpha}}}(n,k)\big\langle[x;\bar{\boldsymbol{\alpha}}|m]_q \big\rangle_k([x]_q-[\alpha_k+km]_q+[\alpha_k+km]_q) \\
&=\sum_{k=0}^{n} W_{q,m,\bar{\boldsymbol{\alpha}}}(n,k)\big\langle[x;\bar{\boldsymbol{\alpha}}|m]_q \big\rangle_{k+1}+\sum_{k=0}^n [\alpha_k+km]_q W_{q,m,\bar{\boldsymbol{\alpha}}}(n,k) \big\langle[x;\bar{\boldsymbol{\alpha}}|m]_q \big\rangle_k\\
&=\sum_{k=1}^{n+1}W_{q,m,\bar{\boldsymbol{\alpha}}}(n,k-1)\big\langle[x;\bar{\boldsymbol{\alpha}}|m ]_q \big \rangle_{k}+\sum_{k=0}^n [\alpha_k+km]_q W_{q,m,\bar{\boldsymbol{\alpha}}}(n,k) \big\langle[x;\bar{\boldsymbol{\alpha}}|m]_q \big\rangle_k.\\
\end{split}
\end{equation*}
Equating the coefficients of $\big\langle[x;\bar{\boldsymbol{\alpha}}|m ]_q \big\rangle_k $ on both sides, we obtain Eq.~(\ref{E:rec2}).

When $k=0$, we get $ W_{q,m,\bar{\boldsymbol{\alpha}}}(n+1,0)= [\alpha_0 ]_q\,W_{q,m,\bar{\boldsymbol{\alpha}}}(n,0), \:n=0 , 1, 2, \ldots$ .
Thus $ W_{q,m,\bar{\boldsymbol{\alpha}}}(n,0)=[\alpha_0]_q^n\, W_{q,m,\bar{\boldsymbol{\alpha}}}(0,0) =[\alpha_0]_q^n$.
\end{proof}
 \begin {definition}
The $(q,\bar{\boldsymbol{\alpha}})$-Whitney matrix of the second kind is the $n\times n$ lower triangular matrix defined by
\[
\mathbf{W}_{2}:=\mathbf{W}_{q,m,\bar{\boldsymbol{\alpha}}}(n):=\big(W_{q,m,\bar{\boldsymbol{\alpha}}}(i,j)\big)_{0 \leq i,j\leq n-1}.
\]
\end{definition}
For example when $n=4 $, the matrix  $\mathbf{W}_{q,m,\bar{\boldsymbol{\alpha}}}(n)$ is given by
\[
\begin{pmatrix}
  1 &0 &0 &0\\ \\
 [\alpha_0]_q & 1 &0 &0 \\ \\
  [\alpha_0]_q^2 &[\alpha_0]_q+[\alpha_1+m]_q & 1& 0 \\ \\
[\alpha_0]_q^3& [\alpha_0]_q^2+[\alpha_0]_q [\alpha_1+m]_q+[\alpha_1+m]_q^2 &[\alpha_0]_q+[\alpha_1+m]_q+[\alpha_2+2m]_q&1\\
\end{pmatrix}
\]
When $q\rightarrow 1$ the matrix $\mathbf{W}_{q,m,\bar{\boldsymbol{\alpha}}}(n)$ is reduced to the $\bar{\boldsymbol{\alpha}}$-Whiyney matrix of the second kind \cite{beih16}, also at $q\rightarrow 1$ and $\bar{\boldsymbol{\alpha}}=(r, r, \ldots, r)$ the $\mathbf{W}_{q,m,\bar{\boldsymbol{\alpha}}}(n)$ is reduced to the $r$-Whitney matrix of the second kind \cite{cheon12, mezo15}.
\begin{theorem}\label{th:w2}
The $(q,\bar{\boldsymbol{\alpha}})$-Whitney numbers of the second kind $W_{q,m,\bar{\boldsymbol{\alpha}}}(n,k)$ have the explicit formula
\begin{equation}\label{E:qwit2}
W_{q,m,\bar{\boldsymbol{\alpha}}}(n,k)=\sum_{j=0}^k \frac{([\alpha_j+jm]_q)^n}{\prod_{i=0 ,i\ne j}^k([\alpha_j+jm]_q-[\alpha_i+im]_q)},
\end{equation}
and satisfy the recurrence relation
\begin{equation}\label{E:rrec2}
\begin{split}
W_{q,m,\bar{\boldsymbol{\alpha}}}(n,k)&= \sum_{j=k}^{n} W_{q,m,\bar{\boldsymbol{\alpha}}}(j-1,k-1) \prod_{i=j}^{n-1}[\alpha_k+km]_q \\
&=\sum_{j=k}^{n} W_{q,m,\bar{\boldsymbol{\alpha}}}(j-1,k-1) \big([\alpha_k+km]_q \big)^{n-j}.
\end{split}
\end{equation}
\end{theorem}
\begin{proof}
Taking $a_i=0$ and $b_i=[\alpha_i+im]_q$ in \eqref{E:mex} and \eqref{E:mrec}, yield \eqref{E:qwit2} and \eqref{E:rrec2}, respectively.
\end{proof}
As $q \rightarrow1$, the recurrence relation \eqref{E:rrec2} reduces to a new recurrence relation for the $\bar{\boldsymbol{\alpha}}$-Whitney numbers of the second kind given by
\begin{equation*}
\begin{split}
W_{m,\bar{\boldsymbol{\alpha}}}(n,k)&= \sum_{j=k}^{n} W_{m,\bar{\boldsymbol{\alpha}}}(j-1,k-1) \prod_{i=j}^{n-1}(\alpha_k+km)\\
&=\sum_{j=k}^{n} W_{m,\bar{\boldsymbol{\alpha}}}(j-1,k-1)(\alpha_k+km)^{n-j}.
\end{split}
\end{equation*}
Using \eqref{E:qwit2} we obtain the exponential generating function of the $(q,\bar{\boldsymbol{\alpha}})$-Whitney numbers of the second kind $W_{q,m,\bar{\boldsymbol{\alpha}}}(n,k)$
\begin{equation}
\begin{split}
\sum _{n=0}^{\infty}W_{q,m,\bar{\boldsymbol{\alpha}}}(n,k)\frac{t^n}{[n]_q!}&= \sum_{j=0}^k \frac{1}{\prod_{i=0 ,i\ne j}^k([\alpha_j+jm]_q-[\alpha_i+im]_q)}\sum _{n=0}^{\infty}\frac{([\alpha_j+jm]_q\, t)^n}{[n]_q!}\\
&=\sum_{j=0}^k \frac{1}{\prod_{i=0 ,i\ne j}^k([\alpha_j+jm]_q-[\alpha_i+im]_q)}e_q([\alpha_j+jm]_q \,t).\\
\end{split}
\end{equation}
\begin{theorem}
The generating function of $W_{q,m,\bar{\boldsymbol{\alpha}}}(n,k)$ is given by
\begin{equation}\label{E:gen2}
Y_{k,q}(t)=\sum_{n=k}^{\infty}W_{q,m, \bar{\boldsymbol{\alpha}}}(n,k)\, t^n=t^k \prod_{j=0}^k(1-[\alpha_j+jm]_q t)^{-1},\:\: \: \: t<\frac{1}{[\alpha_k+km]_q}
\end{equation}
 where $k=1,2,3,\ldots$, and
 \begin{equation}\label{E:gen3}
  Y_{k,q}(0)=0 \:for \: k\geq 1 \: and \: Y_{0,q}(t)=(1-[\alpha_0 ]_q t)^{-1}.
 \end{equation}
 \end{theorem}
 \begin{proof}Equation \eqref{E:gen3} can easily obtained from the definition of generating function
 \[
 Y_{0,q}(t)=\sum_{n=0}^{\infty}W_{q,m, \bar{\boldsymbol{\alpha}}}(n,0) \,t^n=\sum_{n=0}^{\infty}[\alpha_0]_q^n\, t^n=\sum_{n=0}^{\infty}([\alpha_0]_q\, t)^n=(1-[\alpha_0 ]_q\, t)^{-1}.
 \]
From \eqref{E:rec2}, we get
\[
  \sum_{n=k}^{\infty} W_{q,m,\bar{\boldsymbol{\alpha}}}(n,k)\,t^n= \sum_{n=k}^{\infty}W_{q,m,\bar{\boldsymbol{\alpha}}}(n-1,k-1)\,t^n+(\alpha_k+km)\sum_{n=k}^{\infty}W_{q,m,\bar{\boldsymbol{\alpha}}}(n-1,k)\,t^n.
\]
Thus we obtain the recurrence relation for the generating function $Y_{k,q}(t)$
\[
Y_{k,q}(t)=t\:Y_{k-1,q}+[\alpha_k+km]_q\,t\:Y_{k,q}(t),\: \: \: k=1, 2, \ldots \:.
\]
Hence
\begin{equation}\label{E:iter}
Y_{k,q}(t)=\frac{t}{(1-[\alpha_k+km]_q\,t)}\: Y_{k-1,q}(t),\: \: \: k=1, 2, \ldots \:.
\end{equation}
Applying successively this recurrence, we get Eq.~(\ref{E:gen2}).
\end{proof}
The previous theorem shows that the numbers $W_{q, m, \bar{\boldsymbol{\alpha}}}(n,k)$ are the complete symmetric function of the numbers $ [\alpha_0]_q, \: [\alpha_1+m]_q, \: \ldots ,\:[\alpha_k+km]_q $ of order $n-k$.

We obtain  from Eq.~(\ref{E:gen2})
\[
\sum_{n=k}^{\infty}W_{q,m, \bar{\boldsymbol{\alpha}}}(n,k)\, t^{n-k}=\prod_{j=0}^k(1-[\alpha_j+jm]_q t)^{-1},\:\: \: \: t<\frac{1}{[\alpha_k+km]_q}.
\]
Expanding the right hand side and comparing the coefficients of $t^{n-k}$ yields
\[
W_{q, m,\bar{\boldsymbol{\alpha}}}(n,k)=\sum_{ 0 \leq j_1 \leq \cdots \leq j_{n-k} \leq k}\: \prod_{i=1}^{n-k}[\alpha_{j_i}+j_i\:m]_q.
\]
\section {Orthogonality and Inverse Relations}
The orthogonality and the inverse relations for the $\bar{\boldsymbol{\alpha}}$-Whitney numbers of both kinds were obtained in \cite{beih16}. In this section, we establish analogous properties for the $(q,\bar{\boldsymbol{\alpha}})$-Whitney numbers of both kinds.
\begin{theorem}\label{th:orth}
The $(q,\bar{\boldsymbol{\alpha}})$-Whitney numbers of the first and second kind satisfy the following orthogonality relations:
\begin{equation}
\sum_{k=j}^n  W_{q,m,\bar{\boldsymbol{\alpha}}}(n,k)\:w_{q,m,\bar{\boldsymbol{\alpha}}}(k,j)=\delta_{nj},
\end{equation}
and
\begin{equation}
\sum_{k=j}^n  w_{q,m,\bar{\boldsymbol{\alpha}}}(n,k)\:W_{q,m,\bar{\boldsymbol{\alpha}}}(k,j)=\delta_{nj}.
\end{equation}
\end{theorem}
\begin{proof}
Using \eqref{E:wi1} and \eqref{E:wi2} give
\begin{equation*}
\begin{split}
 [x]^n_q&=\sum_{k=0}^n  W_{q,m,\bar{\boldsymbol{\alpha}}}(n,k)\big\langle[x;\bar{\boldsymbol{\alpha}}|m]_q\big\rangle_k \\
&=\sum_{k=0}^n W_{q,m,\bar{\boldsymbol{\alpha}}}(n,k)\sum_{j=0}^k w_{q,m,\bar{\boldsymbol{\alpha}}}(k,j) [x]^j_q \\
&=\sum_{j=0}^n \{\sum_{k=j}^n W_{q,m,\bar{\boldsymbol{\alpha}}}(n,k)w_{q,m,\bar{\boldsymbol{\alpha}}}(k,j)\} [x]^j_q.\\
\end{split}
\end{equation*}
Comparing the coefficients of $ [x]^j_q$ gives
\[
\sum_{k=j}^n W_{q,m,\bar{\boldsymbol{\alpha}}}(n,k)w_{q,m,\bar{\boldsymbol{\alpha}}}(k,j)=\delta_{nj}.
\]
The second relation can be proved similarly.
\end{proof}
The orthogonality properties give the following identities
\[
\mathbf{W}_{2}\; \mathbf{W}_{1}=\mathbf{W}_{1}\; \mathbf{W}_{2}=\mathbf{I}.\: \text{Thus} \: \mathbf{W}_{2}^{-1}= \mathbf{W}_{1}\: \text{and} \: \mathbf{W}_{1}^{-1}= \mathbf{W}_{2}.
\]
The following theorem can easily be deduced from Theorem \ref{th:orth}.
\begin{theorem}\label{th:inv}
The  $(q,\bar{\boldsymbol{\alpha}})$-Whitney numbers of the first and second kind satisfy the following inverse relations
\begin{equation}
f_n=\sum_{k=0}^n W_{q,m,\bar{\boldsymbol{\alpha}}}(n,k)\:g_k\:\:\Longleftrightarrow \:\:g_n=\sum_{k=0}^n w_{q,m,\bar{\boldsymbol{\alpha}}}(n,k)\:f_k.
\end{equation}
\end{theorem}
\begin{proof}
If the condition
\begin{equation*}
f_n=\sum_{k=0}^n W_{q,m,\bar{\boldsymbol{\alpha}}}(n,k)\:g_k
\end{equation*}
holds, then
\begin{equation*}
\begin{split}
\sum_{k=0}^n w_{q,m,\bar{\boldsymbol{\alpha}}}(n,k)\:f_k&=\sum_{k=0}^n w_{q,m,\bar{\boldsymbol{\alpha}}}(n,k)\sum_{m=0}^k W_{q,m,\bar{\boldsymbol{\alpha}}}(k,m)\:g_m\\
&=\sum_{m=0}^k(\sum_{k=m}^nw_{q,m,\bar{\boldsymbol{\alpha}}}(n,k)W_{q,m,\bar{\boldsymbol{\alpha}}}(k,m))g_m.
\end{split}
\end{equation*}
By Theorem \ref{th:orth}, we get
\[
\sum_{k=0}^n w_{q,m,\bar{\boldsymbol{\alpha}}}(n,k)\:f_k=\sum_{m=0}^n \delta_{mn}\:g_m=g_n.
\]
The converse can be shown similarly.
\end{proof}
\textbf{Special cases:}
\begin{enumerate}
\item Setting  $ m=1$ and $\bar{\boldsymbol{\alpha}}=(r, r, \ldots, r):=\mathbf{r}$, then \eqref{E:wi1} and \eqref{E:wi2}, respectively, give
\[
 w_{q,1,\mathbf{r}}(n,k)= s_q(n,k,r) \; \; \text{and}\; \; W_{q,1,\mathbf{r}}(n,k)= S_q(n,k,r),
\]
where $s_q(n,k,r)$ and $S_q(n,k,r)$ are  the non-central q-Stirling numbers of the first and second kind, respectively, see \cite{char03}.
\item Setting  $ m=1$ and $ \bar{\boldsymbol{\alpha}}=(0, 0, \ldots, 0):=\mathbf{0}$, hence \eqref{E:wi1} and \eqref{E:wi2}, respectively, give
 \[
 w_{q,1,\mathbf{0}}(n,k)=s_q(n,k),\; \; \text{and} \; \; W_{q,1,\mathbf{0}}(n,k)=S_q(n,k),
 \]
 where $ s_q(n,k)$ and $ S_q(n,k)$ are the q-Stirling numbers of the first and second kind, respectively, see \cite{carl48, gould61}.
\item Setting $m=1$ and $\alpha_j+j=\beta_j$, for $ j=0, 1, \ldots, n-1$, then \eqref{E:wi1} and \eqref{E:wi2}, respectively, give

\[
w_{q,1,\bar{\boldsymbol{\alpha}}}(n,k)=s_{q,\bar{\boldsymbol{\beta}}}(n,k)\; \; \text{and}\; \; W_{q,1,\bar{\boldsymbol{\alpha}}}(n,k)=S_{q,\bar{\boldsymbol{\beta}}}(n,k),
\]
where $s_{q,\bar{\boldsymbol{\beta}}}(n,k)$ and $S_{q,\bar{\boldsymbol{\beta}}}(n,k)$ are the generalized q-Stirling numbers of the first and second kind
(q-Comtet numbers), respectively, see \cite{beih11}.
\end{enumerate}
\section {The $\bar{\boldsymbol{\alpha}}$-Whitney-Lah numbers }
The signless Lah numbers $ L(n,k)= \frac{n!}{k!}\binom {n-1}{k-1}$ were first studied by Lah \cite{lah54}  and they expressed in terms of the signless stirling numbers $s(n,k)$ of the first kind,  and  the stirling numbers $S(n,k)$ of the second kind
\[
L(n,k)= \sum_{j=k}^{n}s(n,j)S(j,k).
\]
Choen and Jung \cite{cheon12} defined the $r$-Whitney-Lah numbers $L_{m,r}(n,k)$ by
\[
L_{m,r}(n,k)=\sum_{j=k}^{n}(-1)^{n-j} w_{m,r}(n,j)W_{m,r}(j,k).
\]
Analogously, we define the  $\bar{\boldsymbol{\alpha}}$-Whitney-Lah numbers $L_{m,\bar{\boldsymbol{\alpha}}}(n,k) $ as follows:

\begin{equation}\label{E:lah1}
L_{m,\bar{\boldsymbol{\alpha}}}(n,k)=\sum_{j=k}^n (-1)^{n-j}\:w_{m,\bar{\boldsymbol{\alpha}}}(n,j) \: W_{m,\bar{\boldsymbol{\alpha}}}(j,k),
\end{equation}
where $L_{m,\bar{\boldsymbol{\alpha}}}(0,0)=1$ and $L_{m,\bar{\boldsymbol{\alpha}}}(n,k)=0$ for $ n<k$ or $k<0$.
\begin{theorem}
The $\bar{\boldsymbol{\alpha}}$-Whitney-Lah numbers $L_{m,\bar{\boldsymbol{\alpha}}}(n,k) $ may be obtained from
\begin{equation}\label{E:lah}
\prod_{i=0}^{n-1}(x+\alpha_i+im)=\sum_{k=0}^n\:L_{m,\bar{\boldsymbol{\alpha}}}(n,k) (x;\bar{\boldsymbol{\alpha}}|m)_k.
\end{equation}
\end{theorem}
\begin{proof}
Replacing $x$ by $-x$ in Eq.~(\ref{E:wit1}), we get
\begin{equation}
\prod_{i=0}^{n-1}(x+\alpha_i+im)= \sum_{j=0}^{n}\:(-1)^{n-j} w_{m,\bar{\boldsymbol{\alpha}}}(n,j)\: x^j.
\end{equation}
Hence
\begin{equation*}
\begin{split}
\sum_{k=0}^n\:L_{m,\bar{\boldsymbol{\alpha}}}(n,k) (x;\bar{\boldsymbol{\alpha}}|m)_k&=\sum_{k=0}^n\:\sum_{j=k}^n\:(-1)^{n-j}w_{m,\bar{\boldsymbol{\alpha}}}(n,j) W_{m,\bar{\boldsymbol{\alpha}}}(j,k)
(x;\bar{\boldsymbol{\alpha}}|m)_k \\
&=\sum_{j=0}^n\:\sum_{k=0}^j\:(-1)^{n-j}w_{m,\bar{\boldsymbol{\alpha}}}(n,j)\: W_{m,\bar{\boldsymbol{\alpha}}}(j,k)
(x;\bar{\boldsymbol{\alpha}}|m)_k \\
&=\sum_{j=0}^n\:(-1)^{n-j}w_{m,\bar{\boldsymbol{\alpha}}}(n,j)x^j=\prod_{i=0}^{n-1}(x+\alpha_i+im)
\end{split}
\end{equation*}
\end{proof}
\begin{theorem}
 The $\bar{\boldsymbol{\alpha}}$-Whitney-Lah numbers satisfy the recurrence relation
 \begin{equation} \label{E:recl}
 L_{m,\bar{\boldsymbol{\alpha}}}(n+1,k)=L_{m,\bar{\boldsymbol{\alpha}}}(n,k-1)+(\alpha_k+\alpha_n+(k+n)m)L_{m,\bar{\boldsymbol{\alpha}}}(n,k),
 \end{equation}
  where $n\geq k\geq 1$, for $k=0$ we have
 \begin{equation}
 L_{m,\bar{\boldsymbol{\alpha}}}(n,0)=\prod_{i=0}^{n-1}(\alpha_0+\alpha_i+im).
 \end{equation}
\end{theorem}
\begin{proof} We can write
\[\prod_{i=0}^{n}(x+\alpha_i+im)=\prod_{i=0}^{n-1}(x+\alpha_i+im)(x-\alpha_k-km+\alpha_k+km+\alpha_n+nm).
\]
Using \eqref{E:lah}, we get
\begin{equation*}
\begin{split}
&\sum_{k=0}^{n+1} L_{m,\bar{\boldsymbol{\alpha}}}(n+1,k)(x;\bar{\boldsymbol{\alpha}}|m )_k \\
&=\sum_{k=0}^n L_{m,\bar{\boldsymbol{\alpha}}}(n,k)(x;\bar{\boldsymbol{\alpha}}|m )_k \big ((x-\alpha_k-km)+(\alpha_k+km+\alpha_n+nm)\big) \\
&=\sum_{k=0}^{n} L_{m,\bar{\boldsymbol{\alpha}}}(n,k)(x;\bar{\boldsymbol{\alpha}}|m )_{k+1}+\sum_{k=0}^n L_{m,\bar{\boldsymbol{\alpha}}}(n,k) (x;\bar{\boldsymbol{\alpha}}|m )_k (\alpha_k+km+\alpha_n+nm)\\
&=\sum_{k=1}^{n+1} L_{m,\bar{\boldsymbol{\alpha}}}(n,k-1)(x;\bar{\boldsymbol{\alpha}}|m)_{k}+\sum_{k=0}^n  L_{m,\bar{\boldsymbol{\alpha}}}(n,k) (x;\bar{\boldsymbol{\alpha}}|m )_k (\alpha_k+km+\alpha_n+nm).\\
\end{split}
\end{equation*}
Equating the coefficients of $(x;\bar{\boldsymbol{\alpha}}|m )_k $ on both sides, we obtain \eqref{E:recl}.

For $k=0$, we find
 \[
  L_{m,\bar{\boldsymbol{\alpha}}}(n+1,0)=  L_{m,\bar{\boldsymbol{\alpha}}}(n,0) \: (\alpha_0+\alpha_n+nm),\: n=0 , 1, 2, \ldots \: .
 \]
 Consequently, we get
 \[
 L_{m,\bar{\boldsymbol{\alpha}}}(n,0)=L_{m,\bar{\boldsymbol{\alpha}}}(0,0) \: 2 \alpha_0 \: (\alpha_0+\alpha_1+m) \cdots (\alpha_0+\alpha_{n-1}+(n-1)m).
 \]
\end{proof}
\textbf{Special cases:}
\begin{enumerate}
  \item The  $L_{m,\bar{\boldsymbol{\alpha}}}(n,k) $ is reduced to $L(n,k)$ when $ m=1$ and $ \bar{\boldsymbol{\alpha}}=(0, 0, \ldots, 0)$.
  \item The $L_{m,\bar{\boldsymbol{\alpha}}}(n,k) $ is reduced to $L_{m,r}(n,k)$ when  $\bar{\boldsymbol{\alpha}}=(r, r, \ldots, r)$.
  \item The  $L_{m,\bar{\boldsymbol{\alpha}}}(n,k) $ is reduced to  the $r$-Lah numbers $ L_r(n+r,k+r)$ when $ m=1$ and
   $\bar{\boldsymbol{\alpha}}=(r, r, \ldots, r)$, see \cite{nyu15}.
\end{enumerate}
Defining the $\bar{\boldsymbol{\alpha}}$-Whitney-Lah matrix as
\[
\mathbf{L}:= \mathbf{L}_{m,\bar{\boldsymbol{\alpha}}}(n):=\big(L_{m,\bar{\boldsymbol{\alpha}}}(i,j)\big)_{0 \leq i,j\leq n-1}.
\]
For example when $n=4$ the matrix $ \mathbf{L}_{m,\bar{\boldsymbol{\alpha}}}(n)$ is given by
\[
 \left(
 \begin{smallmatrix}
 1 &0 &0 &0\\ \\
2 \alpha_0 & 1 &0 &0 \\ \\
2 \alpha_0(\alpha_0+\alpha_1+m) &2(\alpha_0+\alpha_1+m) & 1&0  \\ \\
2\alpha_0(\alpha_0+\alpha_1+m)(\alpha_0+\alpha_2+2m)& 2(\alpha_0+\alpha_1+m)(\alpha_0+\alpha_1+\alpha_2+3 m) &2(\alpha_0+\alpha_1+\alpha_2+3m )&1\\
\end{smallmatrix}
\right)
\]
In particular, when $\bar{\boldsymbol{\alpha}}=(r, r, \ldots, r)$ the $\mathbf{L}_{m,\bar{\boldsymbol{\alpha}}}(n)$ is reduced to the $r$-Whitney-Lah matrix \cite{mezo15}.
\begin{theorem}\label{th:lah}
The $\bar{\boldsymbol{\alpha}}$-Whitney-Lah numbers $ L_{m,\bar{\boldsymbol{\alpha}}}(n,k)$ have the explicit formula
\begin{equation}
L_{m,\bar{\boldsymbol{\alpha}}}(n,k)=\sum_{j=0}^k \frac{\prod_{i=0}^{n-1}(\alpha_j+jm+\alpha_i+im)}{\prod_{i=0 ,i\ne j}^k (\alpha_j+jm-\alpha_i-im)},
\end{equation}
and the recurrence relation
\begin{equation}
L_{m,\bar{\boldsymbol{\alpha}}}(n,k)=\sum_{j=k}^n L_{m,\bar{\boldsymbol{\alpha}}}(j-1,k-1)\prod_{i=j}^{n-1}(\alpha_i+im+\alpha_k+km).
\end{equation}
\end{theorem}
\begin{proof}
The proof follows by setting $a_i=b_i=\alpha_i+im$ in \eqref{E:mex} and \eqref{E:mrec}.
\end{proof}
In particular, by setting $\bar{\boldsymbol{\alpha}}=(r, r, \ldots, r)$ we obtain the explicit formula and recurrence relation for $L_{m,r}(n,k)$ as follows:
\begin{corollary} The $r$-Whitney-Lah numbers $L_{m,r}(n,k)$ satisfy the following:
\begin{equation}\label{E:rlahe}
L_{m,r}(n,k)=\frac{1}{m^k k!}\sum_{j=0}^k \binom{k}{j}(-1)^{k-j}\prod_{i=0}^{n-1}(2r+jm+im).
\end{equation}
\begin{equation}
L_{m,r}(n,k)=\sum_{j=k}^n L_{m,r}(j-1,k-1)\prod_{i=j}^{n-1}(2r+km+im).
\end{equation}
\end{corollary}
Choen and Jung \cite{cheon12} showed that
\begin{equation}\label{E:choe}
L_{m,r}(n,k)=\binom{n}{k}\prod_{i=0}^{n-k-1}(2r+km+im).
\end{equation}
Thus from \eqref{E:rlahe} and \eqref{E:choe} we obtain the following combinatorial identity
\begin{equation}
\frac{1}{m^k k!}\sum_{j=0}^k \binom{k}{j}(-1)^{k-j}\prod_{i=0}^{n-1}(2r+jm+im)=\binom{n}{k}\prod_{i=0}^{n-k-1}(2r+km+im).
\end{equation}
\subsection {Matrix representations }
Let $\mathbf{w}, \mathbf{W}$ and $\mathbf{L}$ denote infinite lower triangular matrices whose $(n,k)$-th entries are
$w_{m,\bar{\boldsymbol{\alpha}}}(n,k)$, $W_{m,\bar{\boldsymbol{\alpha}}}(n,k)$, and $L_{m,\bar{\boldsymbol{\alpha}}}(n,k)$,
respectively. Furthermore, let $\mathbf{D}$ be the infinite diagonal matrix whose $(n,k)$-th entry is $D(n,k)=(-1)^n \delta{nk}$, hence $\mathbf{D}^{-1}=\mathbf{D}$, and $\mathbf{D} \mathbf{D}^{-1}= \mathbf{I}$.  Equation \eqref{E:lah1} can be written in the matrix form
\[
 \mathbf{L}=\mathbf{D}\mathbf{w}\mathbf{D}\mathbf{W}.
\]
El-Desouky  et al. \cite{beih16} showed that $\mathbf{w}^{-1}=\mathbf{W}$, $\mathbf{W}^{-1}=\mathbf{w}$. Thus
\begin{equation*}
\mathbf{L}^{-1}=\mathbf{W}^{-1}\mathbf{D}\mathbf{w}^{-1}\mathbf{D}=\mathbf{w}\mathbf{D}\mathbf{W}\mathbf{D}=
\mathbf{D}\mathbf{D}\mathbf{w}\mathbf{D}\mathbf{W}\mathbf{D}
=\mathbf{D}\mathbf{L}\mathbf{D}.
\end{equation*}


\begin{thebibliography}{1}
\bibitem{ben96}{\small {\sc M. Benoumhani:} {\it On Whitney numbers of Dowling lattices.} Discrete Math., \textbf{159} (1996), 13--33.}
\bibitem{carl48} {\small {\sc L. Carlitz:} {\it $q$-Bernoulli numbers and polynomials.} Duke Math. J., \textbf{15} (1948), 987--1000.}
\bibitem{char03} {\small {\sc Ch. A. Charalambides:} { \it Non-central generalized $q$-factorial coefcients and $q$-Stirling numbers.}
Discrete Math., \textbf{275} (2004), 67--85.}
\bibitem{cheon12} {\small {\sc G. S. Cheon, J. H. Jung:} {\it $r$-Whitney numbers of Dowling lattices.} Discrete Math., \textbf{312} (2012), 2337--2348.}
\bibitem{beih16} {\small {\sc B. S. El-Desouky, Nenad P. Caki\'c, and F. A. Shiha:} {\it New Family of Whitney Numbers.} Filomat (accepted).}
\bibitem{beih11} {\small {\sc B. S. El-Desouky, R.S. Gomaa:}  {\it $q$-Comtet and generalized $q$-harmonic numbers.} J. Math. Sci. Adv. Appl.,
 \textbf{10} (1/2)(2011), 33--52.}
\bibitem{gould61} {\small {\sc H. W. Gould:} {\it The $q$-Stirling numbers of the first and second kinds.} Duke Math. J., \textbf{28} (1961), 281--289.}
\bibitem{mango15} {\small {\sc M. M. Mangontarum, J. Katriel:} {\it On $q$-boson operators and $q$-analogues of the $r$-Whitney and $r$-Dowling numbers.}
J. Integer Seq., \textbf{18} (2015), Art. 15.9.8.}
\bibitem{mansour12} {\small {\sc T. Mansour, S. Mulay, M. Shattuck:} {\it  A general two-term recurrence and
its solution.} European. J. Combin., \textbf{33}(2012) 20--26.}
\bibitem{merca14}{\small {\sc M. Merca:} {\it A new connection between $r$-Whitney numbers and Bernoulli polynomials.} Integral Transforms Spec.
Funct., \textbf{25}(12) (2014), 937--942.}
\bibitem{mezo2010}{\small {\sc I. Mez\"o:} {\it A new formula for the Bernoulli polynomials.} Results Math., \textbf{58} (2010), 329--335.}
\bibitem{mezo15} {\small {\sc I. Mez\"o, J. L. Ram\'irez:} {\it The linear algebra of the $r$-Whitney matrices.} Integral Transforms Spec. Funct., \textbf{26} (2015), 213--225.}
\bibitem{lah54} {\small {\sc I. Lah:} {\it A new kind of numbers and its application in the actuarial mathematics.} Bol.Inst. Actu\'ar. Port., \textbf{9} (1954), 7--15.}
\bibitem{nyu15}{\small {\sc G. Nyul, G. R´acz:} {\it The $r$-Lah numbers.} Discrete Math., \textbf{338} (2015), 1660--1666.}
\end{thebibliography}
\end{document}